\documentclass[12]{amsart}

\usepackage{geometry}
\usepackage{amsmath,amssymb,color, array}
\newtheorem{thm}{Theorem}

\newtheorem{cor}[thm]{Corollary}
\newtheorem{lem}[thm]{Lemma}

\newtheorem{que}[thm]{Question}
\newtheorem{prop}[thm]{Proposition}

\newtheorem{exas}[thm]{Examples}
\theoremstyle{definition}

\numberwithin{equation}{section}

\newcommand{\Z}{{\mathbb Z}}

\begin{document}

\title[Weakly nil-clean rings]{Rings in which every element is either a sum or \\ a difference of a nilpotent and an idempotent}

\date{\today}

\author{Simion Breaz}
\address{Faculty of Mathematics
and Computer Science, Babe\c s-Bolyai University, 400084 Cluj-Napoca, Romania}
\email{bodo@math.ubbcluj.ro}
\author{Peter Danchev}
\address{Department of Mathematics, Plovdiv University, 4000 Plovdiv, Bulgaria}
\email{pvdanchev@yahoo.com}
\author{Yiqiang Zhou}
\address{Department of Mathematics and Statistics, Memorial University of Newfoundland,
St. John's NFLD A1C 5S7, Canada}
\email{zhou@mun.ca}


\baselineskip=20pt

\begin{abstract}
{Generalizing the notion of nil cleanness from \cite{D13}, in parallel to \cite{DM14}, we define the concept of {\it weak nil cleanness} for an arbitrary ring. Its comprehensive study in different ways is provided as well. A decomposition theorem of a weakly nil-clean ring is obtained. It is completely characterized when an abelian ring is {\it weakly nil-clean}.} It is also completely determined when a matrix ring over a division ring is weakly nil-clean.
\end{abstract}

\maketitle



\section{Introduction and background}

All rings $R$ in this paper are associative with $1$, but not necessarily commutative. The letters $U(R)$, $J(R)$, $Id(R)$, $Nil(R)$ and $Z(R)$ will stand for the set of units, the Jacobson radical, the set of idempotents, the set of nilpotents and the center of $R$, respectively. This work is motivated by the notions of cleanness, weak cleanness, and nil-cleanness. A ring $R$
is called {\it clean} if each element $r\in R$ can be written as $r=u+e$ where $u\in U(R)$ and $e\in Id(R)$ \cite{N77}.
As a variant of a clean ring, a ring is called {\sl nil-clean} if each element $a\in R$ can be represented as $a=b+e$, where $e\in Id(R)$ and $b\in Nil(R)$  \cite{D13}. If this presentation is unique, the ring $R$ is called {\it uniquely nil-clean}. 
Note that a ring $R$ is (nil-)clean if and only if,  for each $a\in R$, $a=b-e$ where $b$ is a unit (resp., nilpotent) and $e$ is an idempotent. In \cite{AA06} the authors called a commutative ring {\sl weakly clean} if every element is a sum or a difference of a unit and an idempotent.
Following the same idea, in \cite{DM14} was stated the definition of a {\it weakly nil-clean} commutative ring as such a ring $R$ for which any element $a\in R$ is of the form $a=b+e$ or $a=b-e$, where $b\in Nil(R)$ and $e\in Id(R)$. It is shown there that weakly nil-clean commutative rings are of necessity clean. Moreover, a ring $R$ is said to be {\it uniquely weakly nil-clean} if each element can be uniquely presented as the sum or the difference of a nilpotent and an idempotent. Likewise, an element $a\in R$ is \textit{nil-clean} (resp., {\it weakly nil-clean}) if $a=b+e$ (resp., $a=b+e$ or $a=b-e$), where $b$ is nilpotent and $e$ is idempotent.

The aim of the current paper is to study weak nil-cleanness in the context of arbitrary but not necessarily commutative rings; resultantly we enlarge the most part of the results in \cite{D13} to this new point of view. It is organized as follows: In the next section we state some fundamental facts concerning weakly nil-clean rings, including a decomposition
theorem (Theorem \ref{description}) which states that every weakly nil-clean ring is a direct product of a nil-clean ring and a ring with the nil-involution property. In the third section we study abelian weakly nil-clean rings, and we prove a structure theorem of these rings (Theorem \ref{fine}). In particular, we show that these rings are exactly the uniquely weakly nil-clean rings. After that, in the fourth section, we completely determine when the matrix ring over a division ring is weakly nil-clean.

 For $a\in R$, we write $\bar a=a+J(R)\in R/J(R)$. We also denote by ${\mathbb M}_n(R)$ the ring of all $n\times n$ matrices over $R$. All other unexplained explicitly below notion and notation are standard and follow essentially those from \cite{L01}.
 
\section{Weakly nil-clean rings}

The following assertion is useful for applications.

\begin{lem}\label{epi} Let $I$ be an ideal of a ring $R$. If $R$ is weakly nil-clean, then $R/I$ is weakly nil-clean. The converse holds if $I$ is nil.
\end{lem}

\begin{proof} We only need to show the converse. Let $a\in R$. Write either $a+I=(b+I)+(e+I)=(b+e)+I$ or $a+I=(b+I)-(e+I)=(b-e)+I$, where $b+I$ is a nilpotent and $e+I$ is an idempotent. It is obvious that $b$ is again a nilpotent. It is well known that
idempotents lift modulo any nil ideal, so that this allows to assume that $e$ is also an idempotent.

Furthermore, $a-b-e \in I$ or $a-b+e \in I$. It follows immediately that $a-e = b+c$ or $a+e=b+c$, where $c\in I$. Either way, $b+c$ is again a nilpotent; in fact, since $b^k=0$ for some $k\in \mathbb{N}$, we have that $(b+c)^k\in I$ because $I\lhd R$ is an ideal. So, $R$ is weakly nil-clean, as desired.
\end{proof}

The next result, which is a reduction to the semiprimitive case, was established in \cite{DM14} for commutative rings only.

\begin{thm}\label{one}
Let $R$ be a ring. The following are equivalent:
\begin{enumerate}
\item $R$ is weakly nil-clean.

\item $6$ is nilpotent and $R/6R$ is weakly nil-clean.

\item $R/J(R)$ is weakly nil-clean and $J(R)$ is nil.
\end{enumerate}
\end{thm}

\begin{proof}
(1)$\Rightarrow$(2). There exist an idempotent $e$ and a nilpotent $b$ such that $2=b+e$ or $2=b-e$. If $2=b+e$, then $2-b=e=e^2=(2-b)^2=4-4b+b^2$; so $2=(3-b)b$ is a nilpotent, and hence is in $J(R)$ as $2$ is central. Thus, $6\in J(R)$ follows. If $2=b-e$, then $b-2=e=e^2=(b-2)^2=b^2-4b+4$; so $6=(5-b)b$ is a nilpotent and is in $J(R)$. Applying Lemma \ref{epi} we obtain the conclusion.

(2)$\Rightarrow$(3). Note that $6\in J(R)$ since $6$ is a central nilpotent element. So $R/J(R)$ is isomorphic
to a factor ring of $R/6R$, hence it is weakly nil-clean in view of Lemma~\ref{epi}.

Suppose $j\in J(R)$. Then $j=b+e$ or $j=b-e$, where $e^2=e\in R$ and $b\in R$ is a nilpotent.
Assume that $b^n=0$. Therefore $(e-j)^n=0$ or $(e+j)^n=0$, and it follows that $e=e^n\in J(R)$. So we deduce $e=0$, and hence $j=b$ is a nilpotent. Thus, $J(R)$ is nil. As an image of $R$, in accordance with Lemma~\ref{epi}, $R/J(R)$ is clearly weakly nil-clean.

(3)$\Rightarrow$(1). Let $a\in R$. Then $\bar a =\bar b+\bar e$ or $\bar a=\bar b-\bar e$, where $\bar e$ is an idempotent and $\bar b$ is a nilpotent. As $J(R)$ is nil, idempotents lift modulo $J(R)$, so we can assume that $e^2=e$. Thus, there exists $j\in J(R)$ such that $a=(b+j)+e$ or $a=(b+j)-e$. Since $\bar b$ is a nilpotent, $b^m\in J(R)$ for some $m>0$. It follows that $(b+j)^m\in J(R)$. As $J(R)$ is nil, the element $(b+j)^m$ is a nilpotent, so $b+j$ is a nilpotent. Hence, $a$ is weakly nil-clean, as needed.
\end{proof}

The finite direct product of nil-clean rings is also a nil-clean ring, but an infinite direct product of nil-clean rings need not be nil-clean (see, for instance, Proposition 3.13 and the Remark after it in \cite{D13}). On the other side, for the case of weakly clean rings it is proved in \cite{AA06} that a direct product of rings is weakly clean if and only if all but one factors are clean and this factor is weakly clean. The same can be said of weak nil cleanness; even more the following criterion holds:

\begin{prop}\label{prod}
Let $R=\prod _{i\in I}R_i$ be a direct product of rings with $|I|\ge 2$.
\begin{enumerate}
\item Suppose $I$ is finite. Then $R$ is weakly nil-clean if and only if there exists $k\in I$ such that $R_k$ is weakly nil-clean and $R_j$ is nil-clean for all $j\not= k$.

\item Suppose $R$ is abelian. Then
$R$ is weakly nil-clean if and only if there exists $k\in I$ such that $R_k$ is weakly nil-clean and $R_j$ is nil-clean for all $j\not= k$ and $(b_i)\in Nil(R)$ whenever $b_i\in Nil(R_i)$ for all $i\in I$.

\end{enumerate}
\end{prop}

\begin{proof}
$(1)$ Suppose that $R$ is weakly nil-clean. As an image of $R$, each $R_i$ is weakly nil-clean
as a consequence of Lemma~\ref{epi}. Assume that there exist two indices $i_1$ and $i_2$ such that
 neither $R_{i_1}$ nor $R_{i_2}$ are nil-clean. Then there exist $r_1\in R_1$ and $r_2\in R_2$ such that
$r_1$ is not a sum of a nilpotent and an idempotent and $r_2$ is not a difference of a nilpotent and an idempotent. Thus $(r_1,r_2)$ is not weakly nil-clean in $R_{i_1}\times R_{i_2}$, a contradiction.

Conversely, we assume that $R_k$ is not nil-clean for a fixed index $k\in I$.
Thus $R_j$ is nil-clean for all $j\neq k$. Let $r=(r_i)\in R$. Then there exist a nilpotent $b_k$ and an idempotent $e_k$ in $R_k$ such that either $r_k=b_1+e_k$ or $r_k=b_1-e_k$.  If $r_k=b_k+e_k$, for each $i\in I\setminus \{k\}$, write $r_i=b_i+e_i$ where $b_i$ is a nilpotent and $e_i$ is an idempotent. Therefore, $r=(b_i)+(e_i)$ is a sum of a nilpotent and an idempotent. If now $r_k=b_k-e_k$, for each $i\in I\backslash \{k\}$, write $r_i=b_i-e_i$ where $b_i$ is a nilpotent and $e_i$ is an idempotent. Consequently, $r=(b_i)-(e_i)$ is a difference of a nilpotent and an idempotent. So, $r$ is weakly nil-clean in $R$, as expected.

(2) The proof is similar to that of (1), except for the part that $R$ being abelian weakly nil-clean implies that
$b:=(b_i)\in Nil(R)$ for any $b_i\in Nil(R_i)$ ($i\in I$). Assume on the contrary that $b$ is not nilpotent. We write $b=c\pm e$ such that $c=(c_i)\in R$ is nilpotent and $e=(e_i)\in R$ is idempotent. Note that we can find a positive integer $k$ such that $c_i^k=0$ for all $i\in I$. Since $b$ is not nilpotent there exists an index $j\in I$ such that $b_j^k\neq 0$. From $c_j=b_j\mp e_j$
we obtain $(b_j\mp e_j)^k=0$. As a consequence of Binomial Theorem we obtain that $e_j$ is nilpotent. This is possible only if $e_j\in \{0,1\}$. But this implies that $b_j^k=0$ or $b_j$ is invertible, a contradiction.
\end{proof}

The following was proved in \cite[Proposition 1.10]{DM14} for commutative rings.

\begin{prop}\label{char-2}
A ring $R$ is nil-clean if and only if  $R$ is weakly nil-clean and $2\in J(R)$.
\end{prop}

\begin{proof}
The necessity follows from \cite[Proposition 3.14]{D13}. For the sufficiency, note that $R/J(R)$ is of characteristic $2$, so $x=-x$ for every $x\in R/J(R)$. Then $R/J(R)$ is a nil-clean ring, and we can apply \cite[Corollary 3.17]{D13} since $J(R)$ is nil. \end{proof}

The main result in this section is the following decomposition theorem of a weakly nil-clean ring, which shows that a nil-clean component is split off from a weakly nil-clean ring. To precisely describe this situation, we need to define a special property of a ring.  According to \cite{GMW07}, a ring is said to satisfy the {\it involution property} if every element is a sum of a unit and an involution (i.e., an element whose square is $1$). Motivated by this, we say that a ring $R$ satisfies the {\it nil-involution property} if, for each $a\in R$, $a=u+v$ where $u\in {\rm Nil}(R)\pm 1$ and $v^2=1$.

\begin{thm}\label{description} The following are equivalent for a ring $R$:
\begin{enumerate}
\item $R$ is a weakly nil-clean ring.
\item $R\cong R_1\times R_2$, where $R_1$ is a nil-clean ring and $R_2$ is $0$ or an indecomposable weakly nil-clean ring with $3\in J(R_2)$. 
\item $R$ is a direct product of a nil-clean ring and a ring with the nil-involution property.
\end{enumerate}
\end{thm}

\begin{proof}
$(1)\Rightarrow (2)$. By Theorem \ref{one}(2), we have $6^n=0$ for some $n>0$. Then $2^nR\cap 3^nR=0$ and $a^nR+3^nR=R$.
So $R\cong (R/2^nR)\times (R/3^nR)$ by the classical Chinese Remainder Theorem. By Lemma \ref{epi}, $R_1$ and $R_2$ are weakly nil-clean. As $2\in J(R_1)$, $R_1$ is nil-clean by Proposition \ref{char-2}. 

We can assume $R_2\not= 0$. Then $3\in J(R_2)$. In particular, $2$ is invertible in $R_2$. Hence we can use \cite[Proposition 3.14]{D13}
to observe that $R_2$ is not nil-clean. By Proposition \ref{prod}, we deduce that every weakly nil-clean ring with $2$ invertible
is indecomposable. Therefore, $R_2$ is indecomposable.

$(2)\Rightarrow (3)$. It suffices to show that any indecomposable weakly nil-clean ring $R$ with $3\in J(R)$ satisfies the nil-involution property. Note that $J(R)$ is nil (see Theorem~\ref{one}) and $3\in J(R)$. Let $a\in R$. Then there exist $b\in {\rm Nil}(R)$ and $e^2=e\in R$ such that $a=b+e$ or $a=b-e$. If $a=b+e$, then $a=((b+3e)-1)+(1-2e)$ with $(1-2e)^2=1$. Moreover, as $b^m=0$ for some $m>0$ and $3\in J(R)$, $(b+3e)^m\in J(R)$, so $b+3e$ is a nilpotent (as $J(R)$ is nil). If $a=b-e$, then $a=((b-3e)+1)+(-1+2e)$ with $(-1+2e)^2=1$. Moreover, as $b^m=0$ for some $m>0$ and $3\in J(R)$, $(b-3e)^m\in J(R)$, so $b-3e$ is a nilpotent (as $J(R)$ is nil). Hence, $R$ satisfies the nil-involution property.

$(3)\Rightarrow (1)$. It suffices to show that any ring $R$ with the nil-involution property is weakly nil-clean by Proposition~\ref{prod}. Since the nil-involution property obviously implies the involution property, $2$ is invertible in $R$ by \cite[Theorem 3.5]{FY08}. Let $a\in R$. Then $-2a=u+v$ where $u\in {\rm Nil}(R)\pm 1$ and $v^2=1$.  If $u=b+1$ with $b\in {\rm Nil}(R)$, then $a=(-b/2)-(1+v)/2$ with
$-b/2\in {\rm Nil}(R)$ and $(1+v)/2$ an idempotent. If $u=b-1$ with $b\in {\rm Nil}(R)$, then $a=(-b/2)+(1-v)/2$ with
$-b/2\in {\rm Nil}(R)$ and $(1-v)/2$ an idempotent. So, $R$ is weakly nil-clean.
\end{proof}

\begin{cor}\label{invol}
A ring $R$ satisfies the nil-involution property if and only if $R$ is weakly nil-clean with $2\in U(R)$.
\end{cor}

\begin{cor}\label{cor7}
Every weakly nil-clean ring is clean.
\end{cor}

\begin{proof}
If $R$ satisfies the nil-involution property, then $R$ has the involution property. So, $R$ is clean by
\cite[Theorem 3.5]{FY08}. Moreover, every nil-clean ring is clean. Therefore, every weakly nil-clean ring is clean by
Theorem~\ref{description}.
\end{proof}

To further understand the structure of a weakly nil-clean ring, we raise the following question.

\begin{que}
Characterize the rings satisfying the nil-involution property. 
\end{que}

Now we give some examples of weakly nil-clean rings. 

\begin{exas}\label{exas9}
{\rm (1) Let $n\geq 2$ and integer. Then the ring $\Z_n$ is weakly nil-clean if and only if $n=2^\ell 3^k$, where $\ell, k\geq 0$ are integers.

(2) Let $n=2^\ell 3^k$, let $R={\mathbb Z}_n[t]/(t^2)$, and let $\sigma : R\rightarrow R$ given by $a+bt\mapsto a$.
It is readily seen that $\sigma $ is an endomorphism of $R$. Set $S=R[x; \sigma ]/(x^2)$, that is, $S=\{r+sx: r,s\in R\}$ with $x^2=0$ and
$xr=\sigma(r)x$ for all $r\in R$. It is easily seen that $J(S)=J(R)+Rx$, so $J(S)^3=0$. Moreover,
$S/J(S)\cong \mathbb Z_n$. Consequently, $S$ is weakly nil-clean by Theorem~\ref{one}.
But for $r=1+t$ we have that $\sigma (r)=1$, so that $xr=\alpha (r)x=x\not= rx$. Thus $S$ is not commutative. Note that $S$ is a local ring. 

(3) Let $R$ be a ring and $M$ an $(R, R)$-bimodule. Then the trivial extension of $R$ by $M$,
$$R\ltimes M=\Big\{\begin{pmatrix} r & x \\ 0& r\end{pmatrix} \mid r\in R,\ x\in M\Big\}$$ is a weakly nil-clean ring if and only if $R$ is a weakly nil-clean ring.

(4) As usual, for any $n\geq 2$, ${\mathbb T}_n(R)$ denotes the triangular matrix ring over $R$. Then  ${\mathbb T}_n(R)$ is weakly nil-clean if and only if $R$ is nil-clean.}
\end{exas}

It is known that the center of a clean ring need not be clean \cite{BR13}.
In that aspect, the following is somewhat surprising, but it is useful in order to approach constructions of weakly nil-clean rings.

\begin{prop}\label{center} The center of a weakly nil-clean ring is again a weakly nil-clean ring.
\end{prop}

\begin{proof} Suppose $R$ is a weakly nil-clean ring. Given $a\in Z(R)$, we write $a=b+e$ or $a=b-e$, where $e\in Id(R)$ and $b\in Nil(R)$. We will foremost deal with the first equality. Multiplying it by $e$ from the left, we have that $ea=e+eb=e(1+b)$. Thus $e=ea(1+b)^{-1}=e(1+b)^{-1}a\in Ra$, and hence it is readily checked that $e\in Ra^k$ for any $k\geq 1$ because $e$ is an idempotent.

Furthermore, multiplying $a=b+e$ by $1-e$ from the left, we write $(1-e)a=(1-e)b$ whence $(1-e)be=(1-e)ae=(1-e)ea=0$ and so $(1-e)b=(1-e)b(1-e)$. Since there is $l\in \mathbb{N}$ such that $b^l=0$, we deduce that $(1-e)a^l=((1-e)a)^l=((1-e)b)^l=(1-e)b^l=0$.

Letting now $y\in R$ be an arbitrary element, by what we have established so far we infer that $ey(1-e)\in Ra^ly(1-e)=Ry(1-e)a^l=0$. Similarly, $(1-e)ye\in (1-e)yRa^l=(1-e)a^lyR=0$. Therefore, $ey(1-e)=(1-e)ye$, i.e., $ey=ye$ which means that $e\in Z(R)$. This ensures that $b\in Z(R)$ and thus it allows us to conclude that $a$ has a nil-clean decomposition.

The same manipulation also works for the other equality $a=b-e$ to get the wanted claim that $Z(R)$ is weakly nil-clean, as asserted.
\end{proof}

\begin{cor} The center of a nil-clean ring is also nil-clean.
\end{cor}

\begin{proof}
If $R$ is a nil-clean ring then $Z(R)$ is a weakly nil-clean ring. But in view of \cite{D13} we have that $2$ is a central nilpotent, hence $2\in J(Z(R))$ and so $Z(R)$ must be nil-clean by Proposition~\ref{char-2}.
\end{proof}

\section{Abelian weakly nil-clean rings}

The following assertion substantially improves Theorems 2.2 from \cite{Ch}.

\begin{thm}\label{fine}
The following are equivalent for a ring $R$:
\begin{enumerate}
\item $R$ is an abelian weakly nil-clean ring.
\item $R\cong R_1\times R_2$, where $R_1$ is abelian with $J(R_1)$ nil such that $R_1/J(R_1)$ is Boolean, and $R_2$ is $0$ or $R_2/J(R_2)\cong\mathbb{Z}_3$ with $J(R_2)$ nil.
\item $R$ is abelian, $J(R)$ is nil, and $R/J(R)$ is isomorphic to either a
Boolean ring, or to $\mathbb Z_3$, or to the direct product of two such rings.

\item $R$ is a uniquely weakly nil-clean ring.
\end{enumerate}
\end{thm}

\begin{proof}
(1)$\Rightarrow$(2). By Theorem \ref{one}, $J(R)$ is nil.
Employing Theorem \ref{description}, we write $R\cong R_1\times R_2$, where $R_1$ is an abelian nil-clean ring and $R_2$ is $0$ or an abelian indecomposable weakly nil-clean ring with $3\in J(R_2)$. By \cite[Proposition 3.18 and Corollary 3.19]{D13}, $R_1/J(R_1)$ is Boolean. Assume $R_2\not= 0$. We observe that $\mathrm{Id}(R_2)=\{0,1\}$, so every element of $R_2$ is nilpotent or invertible. Therefore, $R_2$ has to be local. Thus, $R_2/J(R_2)$ is a division ring of characteristic $3$ and is weakly nil-clean. It must be that $R_2/J(R_2)\cong \mathbb Z_3$.

(2)$\Rightarrow$(3). It is obvious.

(3)$\Rightarrow$(4). By Theorem~\ref{one}, $R$ is weakly nil-clean. As $J(R)$ is nil and $R/J(R)$ is reduced, we deduce $J(R)=Nil(R)$. Assume that, for $a\in R$, there exist idempotents $e$ and $f$ and nilpotents $b$ and $c$ such that $a=b+e$ or $a=b-e$ and that $a=c+f$ or $a=c-f$. We have to show that $e=f$. There are four cases that we have to consider:
\begin{itemize}
\item[(i)] $a=b+e=c+f$; \item[(ii)] $a=b+e=c-f$; \item[(iii)] $a=b-e=c+f$; \item[(iv)] $a=b-e=c-f$.\end{itemize}
For case (i) or (iv), we have $e-f\in J(R)$. For case (ii) or (iii), we have $e+f\in J(R)$. Thus, in any case,
we have $e-f=(e-f)(e+f)\in J(R)$. It follows that $(1-e)f=-(1-e)(e-f)\in J(R)$ and $e(1-f)=(e-f)(1-f)\in J(R)$. As both $(1-e)f$ and $e(1-f)$ are idempotents, we conclude that $(1-e)f=0$ and $e(1-f)=0$. So $f=ef=e$, as required.

(4)$\Rightarrow$(1). It is enough to prove that $R$ is abelian. Let $e^2=e\in R$. Then, for any $r\in R$,
$e=e+0=(e-er(1-e))+er(1-e)$ are two decompositions into  the sum of an idempotent and a nilpotent. So
$er(1-e)=0$. Similarly, $(1-e)re=0$. Hence $er=re$, and thus all idempotents in $R$ are central, that is,
$R$ is abelian.
\end{proof}

Utilizing the same technique, a part of the last statement can be slightly extended to the quotient $R/J(R)$.

\begin{cor}
Let $R/J(R)$ be an abelian ring. Then $R$ is weakly nil-clean if and only if $J(R)$ is nil and $R/J(R)$ is isomorphic to either a Boolean ring, or to $\mathbb Z_3$, or to the direct product of two such rings.
\end{cor}

The next assertion settles in the affirmative Problem 1 from \cite{DM14}. 
It is worth noting that an abelian weakly nil-clean ring need not be commutative by Example \ref{exas9}(2).

\begin{cor}\label{commut}
Any reduced weakly nil-clean ring is commutative.
\end{cor}

As established above in Corollary \ref{cor7}, every weakly nil-clean ring is always clean. However, by a combination of Theorem~\ref{fine} with results from \cite{NZ04}, one can infer the surprising fact that a uniquely weakly nil-clean ring is not necessarily uniquely clean, and hence not necessarily uniquely nil-clean.

We recall from \cite[Theorem 5.4]{Ch} that {\it a ring $R$ is uniquely nil-clean if and only if $R$ is abelian nil-clean}. So, with Theorem~\ref{fine} at hand, one can deduce the following analogue of Proposition~\ref{char-2}.

\begin{cor} A ring $R$ is uniquely nil-clean if and only if $R$ is uniquely weakly nil-clean and $2\in J(R)$.
\end{cor}

As aforementioned, it was proved in \cite{DM14} that a commutative weakly nil-clean ring is clean. Here we shall extend this to abelian rings in a wider context. An element $a$ in a ring $R$ is called {\it strongly $\pi$-regular} if $a^n\in Ra^{n+1}\cap a^{n+1}R$ for some $n>0$, and the ring is called {\it strongly $\pi$-regular} if every element of $R$ is strongly $\pi$-regular. It is known that every strongly $\pi$-regular ring is clean (see \cite{BM88} or \cite{N99}).

\begin{prop}\label{pireg}
Every weakly nil-clean element of an abelian ring is strongly $\pi$-regular. In particular, every abelian weakly nil-clean ring is strongly $\pi$-regular.
\end{prop}

\begin{proof}
Suppose that $r$ is a weakly nil-clean element of an abelian ring $R$.  If $r = -e + b$ for an idempotent $e$ and a nilpotent $b$, then $r = (1 - e) + (b -1)$, which is a decomposition of $r$ into the sum of an idempotent and a unit.  If, on the other hand, $r = b+e$ for an idempotent $e$ and a nilpotent $b$, then we can write $r = (1 - e) + (2e - 1 + b)$, which is also a decomposition of $r$ into the sum of an idempotent and a unit.  In fact, in each of these cases, $r(1-e)=b(1-e)$ is a nilpotent.  So $r$ is strongly $\pi$-regular by \cite[Proposition 2.5]{D13}.
\end{proof}

As a connection to strongly $\pi$-regular rings, one may state the following strengthening of results on uniquely nil-cleanness of rings from \cite{C11} and \cite{D13}.

\begin{cor} A ring $R$ is uniquely weakly nil-clean if and only if $R$ is abelian strongly $\pi$-regular such that $R/J(R)$ is isomorphic to either a Boolean ring, or to $\mathbb Z_3$, or to the direct product of two such rings.
\end{cor}

\begin{proof} It is well known that strongly $\pi$-regular rings $R$ have nil $J(R)$. Henceforth, we employ Proposition~\ref{pireg} and Theorem~\ref{fine} to get the wanted claim.
\end{proof}

\section{When is ${\mathbb M}_n(R)$ weakly nil-clean?}

By \cite{HN01}, matrix rings over a clean ring are again clean as well as by \cite[Proposition 2.6]{N77}. It is still a left-open question whether the matrix ring over a nil-clean ring is
nil-clean (see \cite{D13}). Note that in \cite{BCDM13} this was settled in the affirmative provided that
$R$ is commutative. However, as it will be manifestly shown below, even in the commutative case, the matrix
ring over a weakly nil-clean ring surprisingly need not be weakly nil-clean.

We will completely determine when a matrix ring over a division ring is weakly nil-clean. Let us start with a reduction technical assertion.

\begin{lem}\label{similar} Suppose that $A$ and $-A$ are similar matrices. Then $A$ is nil-clean if and only if $A$ is weakly nil-clean.
\end{lem}

\begin{proof}
If $A=N-E$ where $N$ is nilpotent and $E$ is idempotent and $P$ is an invertible  matrix such that $-A=PAP^{-1}$, then it is easy to check that $A=(-PNP^{-1})+PEP^{-1}$, that $-PNP^{-1}$ is a nilpotent matrix, and that $PEP^{-1}$ is idempotent. So, $A$ is nil-clean. The converse implication is obvious.
\end{proof}

The following statement is known, but we include a proof for the sake of completeness and reader's convenience.

\begin{lem}\label{trace=rank}
Let $F$ be a field and $E^2=E\in {\mathbb M}_n(F)$. Then ${\rm trace}(E)={\rm rank}(E)\cdot 1_F$.
\end{lem}
\begin{proof}
It is known that $E$ is similar to a diagonal matrix, so there exists an invertible matrix $P\in {\mathbb M}_n(F)$ such that
$PEP^{-1}=\begin{pmatrix}I_k&\bf 0\\
                       \bf 0&\bf 0\end{pmatrix}$ with $k={\rm rank}(E)$. Because the trace is similarity-invariant, we have
											${\rm trace}(E)={\rm trace}(PEP^{-1})=k\cdot 1_F={\rm rank}(E)\cdot 1_F$.
\end{proof}

\begin{lem}\label{2x2}
The matrix $A=\begin{pmatrix}
               1 & 0  \\
               0 & -1  \\
               \end{pmatrix}\in\mathbb{M}_3(\Z_3)$
is not nil-clean.
\end{lem}

\begin{proof}
Assume $A=N+E$, where $N$ is nilpotent and $E$ is idempotent. There exists an invertible matrix $P\in {\mathbb M}_2(\mathbb Z_3)$ such that
$PNP^{-1}=\begin{pmatrix}
               0 & x  \\
               0 & 0  \\
               \end{pmatrix}$. From $PAP^{-1}=PNP^{-1}+PEP^{-1}$, it follows with Lemma~\ref{trace=rank} at hand that $0={\rm trace}(A)={\rm trace}(PAP^{-1})={\rm trace}(PNP^{-1})+{\rm trace}(PEP^{-1})=0+{\rm trace}(E)={\rm rank}(E)\cdot 1_{\mathbb Z_3}$, which is possible only if ${\rm rank}(E)=0$. Thus, we have $E=0$ , a contradiction.
\end{proof}

As a direct consequence to Lemmas~\ref{similar} and \ref{2x2}, we derive:

\begin{cor}
The matrix ring ${\mathbb M}_2(\mathbb Z_3)$ is not weakly nil-clean, though $\mathbb Z_3$ is weakly nil-clean.
\end{cor}

\begin{lem}\label{3x3}
The matrix $A=\begin{pmatrix}
               1 & 0 & 0 \\
               0 & -1 & 0 \\
               0 & 0 & 0
              \end{pmatrix}\in\mathbb{M}_3(\Z_3)$
is not nil-clean.
\end{lem}

\begin{proof}
Assume that $A=N+E$, where $N$ is nilpotent and $E$ is idempotent. As argued in the proof of Lemma \ref{2x2}, we obtain ${\rm rank}(E)\cdot 1_{\mathbb Z_3}=0$, which is possible only if ${\rm rank}(E)=0$ or ${\rm rank}(E)=3$. Hence $E=0$ or $E=I_3$ (as $E$ is idempotent). It follows that $A$ is nilpotent or invertible, a contradiction.
\end{proof}

If $A$ is a matrix, we denote by $L_i^A$ the $i$-th row of $A$ and by $C_j^A$ the $j$-th column of $A$. The next technical claim is a crucial tool.

\begin{lem}\label{nonil}
Let $A=\begin{pmatrix}
               1 & 0 & 0 & 0 \\
               0 & -1 & 0 & 0\\
               0 & 0 & 0 & 0 \\
               0 & 0 & 0 & 0
              \end{pmatrix}\in\mathbb{M}_4(\Z_3).$
If $A=E+N$ with $E$ an idempotent matrix and $N$ a nilpotent matrix, then $N\not= \begin{pmatrix}
               x_{11} & x_{12} & x_{13} & x_{14} \\
               x_{21} & x_{22} & x_{23} & x_{24}\\
               1 & 0 & x_{33} & x_{34} \\
               0 & 1 & x_{43} & x_{44}
      \end{pmatrix}.$
 \end{lem}
\begin{proof}
Assume that $A=E+N$ such that $E=(e_{ij})$ is idempotent and $$N=\begin{pmatrix}
               x_{11} & x_{12} & x_{13} & x_{14} \\
               x_{21} & x_{22} & x_{23} & x_{24}\\
               1 & 0 & x_{33} & x_{34} \\
               0 & 1 & x_{43} & x_{44}
      \end{pmatrix}$$
			is nilpotent. Then we see
$$ N=\begin{pmatrix}
               1-e_{11} & -e_{12} & -e_{13} & -e_{14} \\
               -e_{21} & -1-e_{22} & -e_{23} & -e_{24}\\
               1 & 0 & -e_{33} & -e_{34} \\
               0 & 1 & -e_{43} & -e_{44}
      \end{pmatrix}.$$
			
\noindent From the equalities
$NE=AE-E$, $EN=EA-E$, and $A^2=E+EN+NE+N^2$ we obtain that

  $$    N^2=\begin{pmatrix}
               1-e_{11} & e_{12} & 0 & 0 \\
               e_{21} & 1 & -e_{23} & -e_{24}\\
               0 & 0 & e_{33} & e_{34} \\
               0 & 1 & e_{43} & e_{44}
      \end{pmatrix}.$$
Since $N^4=0$ we have $L_4^{N^2}C_1^{N^2}=0$, hence $e_{21}=0$. We also have $L_2^{N^2}C_2^{N^2}=0$, hence $e_{24}=1$.

Finally, from $L_2^{N}C_2^{N}=1$ we have $e_{21}e_{12}+(-1-e_{22})^2-e_{24}=1$. That is, $(1+e_{22})^2=-1$.
But the equation $x^2=-1$ is not solvable in $\Z_3$, so we have the expected contradiction.
\end{proof}

We come to our basic result in this section describing when the full matrix ring over a division ring is weakly nil-clean. It completely exhausts Problem 2 from \cite{DM14}. Before doing that, we need the following useful technicality.

\begin{lem}\label{power}
Let $D$ be a division ring and $n\ge 1$. If ${\mathbb M}_n(D)$ is weakly nil-clean, then $|D|\le 3$.
\end{lem}

\begin{proof}
It is easy to see that $D$ is weakly nil-clean if and only if $|D|\le 3$. So we can assume $n\ge 2$.
Let $A=\begin{pmatrix}a&0&\cdots &0\\
                       0&0&\cdots &0\\
                       \vdots&\vdots&\ddots&\vdots\\
                       0&0&\cdots&0\end{pmatrix}\in {\mathbb M}_n(D)$, where $a\in D\backslash \{0, 1, -1\}$. By adapting the proof of \cite[Theorem 3]{KLZ14}, we deduce that $A$ is not weakly nil-clean in ${\mathbb M}_n(D)$. Hence ${\mathbb M}_n(D)$ being weakly nil-clean implies $|D|\le 3$, as claimed.
\end{proof}

\begin{thm}\label{main}
Let $D$ be a division ring and $n\ge 1$. Then ${\mathbb M}_n(D)$ is weakly nil-clean if and only if either
\begin{enumerate} \item $D\cong \mathbb Z_2$, or
\item $D\cong \mathbb Z_3$ and $n=1$.
\end{enumerate}
\end{thm}
\begin{proof}
$(\Leftarrow$). By the usage of \cite{BCDM13}, the ring ${\mathbb M}_n(\mathbb Z_2)$ is nil-clean, and hence it is immediately weakly nil-clean. Moreover, it is obviously seen that $\mathbb Z_3$ is weakly nil-clean.

$(\Rightarrow )$. Suppose  ${\mathbb M}_n(D)$ is weakly nil-clean. It follows from Lemma~\ref{power} that $|D|\le 3$. To finish the proof, we assume
$D=\mathbb Z_3$ and then verify $n=1$. Let $A_{11}=\begin{pmatrix}1&0\\
                       0&-1\end{pmatrix}\in {\mathbb M}_2(\mathbb Z_3)$, $A_{1}=\begin{pmatrix}1&0&0\\
                       0&-1&0\\
											0&0&0\end{pmatrix}\in {\mathbb M}_3(\mathbb Z_3)$, and $A_{2}=\begin{pmatrix}1&0&0&0\\
                       0&-1&0&0\\
											0&0&0&0\\
											0&0&0&0\end{pmatrix}\in {\mathbb M}_4(\mathbb Z_3)$.  Next we show that  $A:=\begin{pmatrix}A_{11}&\bf 0\\
                       \bf 0&\bf 0\end{pmatrix}\in{\mathbb M}_n(\mathbb Z_3)$ is not weakly nil-clean for all $n\ge 2$. Let $Q=\begin{pmatrix}V&\bf 0\\
                       \bf 0&I_{n-2}\end{pmatrix}$ with $V=\begin{pmatrix}0&1\\
                       1&0\end{pmatrix}$. Then $QAQ^{-1}=-A$. Thus, by Lemma \ref{similar}, it suffices to show that $A\in{\mathbb M}_n(\mathbb Z_3)$ is not nil-clean for all $n\ge 2$. In view of Lemmas \ref{2x2} and \ref{3x3}, $A$ is not nil-clean in  ${\mathbb M}_n(\mathbb Z_3)$ for $2\leq n\leq 3$.
											
Assume that, for some $n\ge 4$, $A\in {\mathbb M}_n(\mathbb Z_3)$ is nil-clean.
Then, there exists a nilpotent $B=\begin{pmatrix}B_{11}&B_{12}\\
                       B_{21}&B_{22}\end{pmatrix}\in{\mathbb M}_n(\mathbb Z_3)$  with $B_{11}\in {\mathbb M}_2(\mathbb Z_3)$ such that
$A-B$ is an idempotent.	We show that this will lead to a contradiction.

As $B_{21}$ has $2$ columns, we infer that ${\rm rank}(B_{21})\le 2$.
If ${\rm rank}(B_{21})=0$, then $B_{21}=0$. It follows that $B_{11}$ is a nilpotent and that $A_{11}-B_{11}$ is an idempotent, and this is a contradiction by Lemma \ref{2x2}.

If ${\rm rank}(B_{21})>0$, the Gauss elimination shows that there exists an invertible matrix $U$ such that $UB_{21}$ is a reduced row echelon matrix.  Consider the invertible matrix $P=\begin{pmatrix}I_2&0\\
                       0&U\end{pmatrix}$. Then $PAP^{-1}=A$ and $B':=PBP^{-1}=\begin{pmatrix}B_{11}&B_{12}U^{-1}\\
                       UB_{21}&UB_{22}U^{-1}\end{pmatrix}$ is a nilpotent such that $A-B'$ is an idempotent. There are two cases as addressed below. Write $B'=(b_{ij})$.
											
{\bf Case 1}: If ${\rm rank}(B_{21})=1$, then
$UB_{21}=\begin{pmatrix}1&x\\
                       \bf 0&\bf 0\end{pmatrix}$ or $UB_{21}=\begin{pmatrix}0&1\\
                       \bf 0&\bf 0\end{pmatrix}$. From that  $A-B'$ is an idempotent, it follows that $b_{43}=\cdots=b_{n3}=0$. Hence
$B_1:=\begin{pmatrix}b_{11}&b_{12}&b_{13}\\
                       b_{21}&b_{22}&b_{23}\\
                       b_{31}&b_{32}&b_{33}\end{pmatrix}$ is a nilpotent and $A_1-B_1$ is an idempotent.	 This is a contradiction by Lemma \ref{3x3}.

{\bf Case 2}: $UB_{21}=\begin{pmatrix}1&0\\
                         0&1\\
                       \bf 0&\bf 0\end{pmatrix}$. From that $A-B'$ is an idempotent, it follows that $b_{53}=\cdots=b_{n3}=0$ and $b_{54}=\cdots=b_{n4}=0$. Hence, we deduce that $B_2:=\begin{pmatrix}b_{11}&b_{12}&b_{13}&b_{14}\\
                       b_{21}&b_{22}&b_{23}&b_{24}\\
                       1&0&b_{33}&b_{34}\\
											0&1&b_{43}&b_{44}\end{pmatrix}$ is a nilpotent and $A_2-B_2$ is an idempotent.	This is a contradiction by Lemma~\ref{nonil}.																																 
\end{proof}

Recall that a ring $R$ is {\it semilocal} if $R/J(R)$ is semisimple Artinian. As a direct consequence of Theorems \ref{one}, \ref{description} and \ref{main}, accomplished with the classical Wedderburn-Artin Theorem, we immediately yield:

\begin{cor} The following are equivalent for a semilocal ring $R$:
\begin{enumerate}
\item $R$ is weakly nil-clean;
\item  $J(R)$ is nil and $R/J(R)$ is isomorphic to either $C$, or $\mathbb Z_3$, or $C\times \mathbb Z_3$, where $C$ is a finite direct product of matrix rings over $\mathbb Z_2$.
\item $R\cong R_1\times R_2$, where $J(R_1), J(R_2)$ are nil, $R_1/J(R_1)$ is $0$ or a finite direct product of matrix rings over $\mathbb Z_2$, and $R_2/J(R_2)$ is $0$ or isomorphic to ${\mathbb Z}_3$.
\end{enumerate}
\end{cor}

The following strengthens Corollary 4 of \cite{KLZ14} and also uses another idea for proof. Remember that a ring $R$ is said to be {\it strongly regular} if it is an abelian (von Neumann) regular ring.  By Theorem \ref{fine},  strongly regular weakly nil-clean ring is isomorphic to either a Boolean, or $\Z_3$, or a direct product of two such rings.

\begin{cor}\label{bool} Let $R$ be a strongly regular ring and let $n\geq 2$. Then $\mathbb{M}_n(R)$ is weakly nil-clean if and only if $R$ is Boolean.
\end{cor}

\begin{proof} The sufficiency follows directly from \cite{BCDM13}, so that let we now treat the necessity. It is well known that every strongly regular ring is a subdirect product of division rings (see, e.g., \cite{L01}). Then $\mathbb{M}_n(R)$ is a subdirect product of matrix rings over division rings (cf. \cite{L01}). By virtue of Lemma \ref{epi}, we deduce that every such a matrix ring is weakly nil-clean, hence Theorem~\ref{main} allows us to conclude that every division ring is isomorphic to $\Z_2$. Thus $R$ must be Boolean, as asserted.
\end{proof}

\section*{Acknowledgments} The first-named author is supported by the CNCS-UEFISCDI grant PN-II-RU-PCE-2012-4-0100. The research of the third-named author was supported by a Discovery Grant (grant \# 194196) from NSERC of Canada.

\end{document}